\documentclass[12pt]{amsart}

\usepackage{amsmath, amsthm, amssymb, amsfonts, verbatim}

\def\0{{\bf 0}}

\def\R{{\mathbb R}}
\def\Z{{\mathbb Z}}

\theoremstyle{plain}

\newtheorem{thm}{Theorem}

\newtheorem{prop}[thm]{Proposition}
\newtheorem{lem}[thm]{Lemma}

\newtheorem{theorem}{Theorem}
\newtheorem{lemma}[thm]{Lemma}

\theoremstyle{definition}
\newtheorem*{definition}{Definition}

\newtheorem*{theorem*}{Theorem}

\theoremstyle{remark}
\newtheorem{remark}[equation]{Remark}

\begin{document}

\raggedbottom

\numberwithin{equation}{section}

%
%
\newcommand{\MarginNote}[1]{
    \marginpar{
        \begin{flushleft}
            \footnotesize #1
        \end{flushleft}
        }
    }
%
%
\newcommand{\NoteToSelf}[1]{
    }

%
%
\newcommand{\Obsolete}[1]{
    }

\newcommand{\Detail}[1]{
    \MarginNote{Detail}
    \skipline
    \hspace{+0.25in}\fbox{\parbox{4.25in}{\small #1}}
    \skipline
    }

\newcommand{\Todo}[1]{
    \skipline \noindent \textbf{TODO}:
    #1
    \skipline
    }

\newcommand{\Comment}[1] {
    \skipline
    \hspace{+0.25in}\fbox{\parbox{4.25in}{\small \textbf{Comment}: #1}}
    \skipline
    }

%
%

\newcommand{\IntTR}
    {\int_{t_0}^{t_1} \int_{\R^d}}

\newcommand{\IntAll}
    {\int_{-\iny}^\iny}

\newcommand{\Schwartz}
    {\ensuremath \Cal{S}}

\newcommand{\SchwartzR}
    {\ensuremath \Schwartz (\R)}

\newcommand{\SchwartzRd}
    {\ensuremath \Schwartz (\R^d)}

\newcommand{\SchwartzDual}
    {\ensuremath \Cal{S}'}

\newcommand{\SchwartzRDual}
    {\ensuremath \Schwartz' (\R)}

\newcommand{\SchwartzRdDual}
    {\ensuremath \Schwartz' (\R^d)}

\newcommand{\HSNorm}[1]
    {\norm{#1}_{H^s(\R^2)}}

\newcommand{\HSNormA}[2]
    {\norm{#1}_{H^{#2}(\R^2)}}

\newcommand{\Holder}
    {H\"{o}lder }

\newcommand{\Holders}
    {H\"{o}lder's }

\newcommand{\Holderian}
    {H\"{o}lderian }

\newcommand{\HolderRNorm}[1]
    {\widetilde{\Vert}{#1}\Vert_r}

\newcommand{\LInfNorm}[1]
    {\norm{#1}_{L^\iny(\Omega)}}

\newcommand{\SmallLInfNorm}[1]
    {\smallnorm{#1}_{L^\iny}}

\newcommand{\LOneNorm}[1]
    {\norm{#1}_{L^1}}

\newcommand{\SmallLOneNorm}[1]
    {\smallnorm{#1}_{L^1}}

\newcommand{\LTwoNorm}[1]
    {\norm{#1}_{L^2(\Omega)}}

\newcommand{\SmallLTwoNorm}[1]
    {\smallnorm{#1}_{L^2}}

\newcommand{\LpNorm}[2]
    {\norm{#1}_{L^{#2}}}

\newcommand{\SmallLpNorm}[2]
    {\smallnorm{#1}_{L^{#2}}}

\newcommand{\lOneNorm}[1]
    {\norm{#1}_{l^1}}

\newcommand{\lTwoNorm}[1]
    {\norm{#1}_{l^2}}

\newcommand{\MsrNorm}[1]
    {\norm{#1}_{\Cal{M}}}

\newcommand{\FTF}
    {\Cal{F}}

\newcommand{\FTR}
    {\Cal{F}^{-1}}

\newcommand{\InvLaplacian}
    {\ensuremath{\widetilde{\Delta}^{-1}}}

\newcommand{\EqDef}
    {\hspace{0.2em}={\hspace{-1.2em}\raisebox{1.2ex}{\scriptsize def}}\hspace{0.2em}}

%
%

%
%

\title
    [Vanishing viscosity for nondecaying initial data]
    {Vanishing viscosity in the plane for nondecaying velocity and vorticity}

\author{Elaine Cozzi}
\address{Department of Mathematical Sciences, Carnegie Mellon University}
\curraddr{}
\email{ecozzi@andrew.cmu.edu}

\date{} 



\begin{abstract}
Assuming that initial velocity and initial vorticity are bounded in the plane, we show that on a sufficiently short time interval the unique solutions of the Navier-Stokes equations converge uniformly to the unique solution of the Euler equations as viscosity approaches zero.  We also establish a rate of convergence.
\end{abstract}

\maketitle

\section{Introduction}
\noindent We consider the Navier-Stokes equations modeling incompressible viscous fluid flow, given by
\begin{align*}
    \begin{matrix}
        (NS) & \left\{
            \begin{matrix}
                \partial_t v_{\nu} + v_{\nu} \cdot \nabla v_{\nu} - \nu \Delta
                 v_{\nu} = - \nabla p_{\nu} \\
                \text{div } v_{\nu} = 0 \\
                v_{\nu}|_{t = 0} = v_{\nu}^0,
            \end{matrix}
            \right.
    \end{matrix}
\end{align*}
and the Euler equations modeling incompressible non-viscous fluid flow, given by
\begin{align*}
    \begin{matrix}
        (E) & \left\{
            \begin{matrix}
                \partial_t v + v \cdot \nabla v = - \nabla p\\
                \text{div } v = 0 \\
                v|_{t = 0} = v^0.
            \end{matrix}
            \right.
    \end{matrix}
\end{align*}
In this paper, we study the vanishing viscosity limit.  The question of vanishing viscosity addresses whether or not a solution $v_{\nu}$ of ($NS$) converges in some norm to a solution $v$ of ($E$) with the same initial data as viscosity tends to $0$.  This area of research is active both for solutions in a bounded domain and for weak solutions in the plane.  We focus our attention on the latter case.  

The vanishing viscosity problem is closely tied to uniqueness of solutions to the Euler equations, because the methods used to prove uniqueness can often be applied to show vanishing viscosity.  One of the most important uniqueness results in the plane is due to Yudovich.  He establishes in \cite{Yudovich} the uniqueness of a solution $(v,p)$ to ($E$) in the space $C(\R; L^2(\R^2))\times L^{\infty}_{loc}(\R;L^2(\R^2))$ when $v^0$ belongs to $L^2(\R^2)$ and $\omega^0$ belongs to $L^p(\R^2)\cap L^{\infty}(\R^2)$ for some $p<\infty$.  For this uniqueness class, Chemin proves in \cite{Chemin} that the vanishing viscosity limit holds in the $L^p$-norm, and he establishes a rate of convergence.  (In fact, the author only considers the case $p=2$; however, the proof of the result can easily be generalized to any $p<\infty$.)  

In this paper, we consider the case where initial velocity and initial vorticity are bounded and do not necessarily belong to $L^p(\R^2)$ for any $p<\infty$.  The existence and uniqueness of solutions to ($NS$) without any decay assumptions on the initial velocity is considered by Giga, Inui, and Matsui in \cite{GIM}.  The authors establish the short-time existence and uniqueness of mild solutions $v_{\nu}$ to ($NS$) in the space $C([0,T_0]; BUC(\R^n))$ when initial velocity is in $BUC(\R^n)$, $n\geq 2$.  Here $BUC(\R^n)$ denotes the space of bounded uniformly continuous functions on $\R^n$ (see Theorem \ref{GMS} for details).  In \cite{GMS}, Giga, Matsui, and Sawada prove that when $n=2$, the unique solution can be extended globally in time.

Under the assumption that both initial velocity and initial vorticity belong to $L^{\infty}(\R^2)$, Serfati shows in \cite{Serfati} that a unique weak solution $v$ to ($E$) exists in $L^{\infty}([0,T]; L^{\infty}(\R^2))$ (see Theorem \ref{Serfati}).

We prove that the vanishing viscosity limit holds for short time in the $L^{\infty}$-norm when initial velocity and initial vorticity belong to $L^{\infty}(\R^2)$ (see Theorem \ref{main}).  To establish the result, we consider low and high frequencies of the difference between the the solutions to ($NS$) and ($E$) separately.  For low frequencies, we utilize the structure of mild solutions to ($NS$) as well as the structure of Serfati solutions to ($E$).  For high frequencies, we localize the frequencies of the vorticity formulations of ($NS$) and ($E$), and we consider the difference of the two resulting equations.  We make use of the Littlewood-Paley operators and Bony's paraproduct decomposition to prove the necessary estimates.
\section{A Few Definitions and Technical Lemmas}
We first define the Littlewood-Paley operators.  We let $\varphi \in S(\R^n)$ satisfy supp $\varphi\subset \{\xi\in \R^n: \frac{3}{4} \leq|\xi |\leq \frac{8}{3} \}$, and for every $j\in\Z$ we let $\varphi_j(\xi)=\varphi(2^{-j}\xi)$ (so
${\check{\varphi}}_j(x)=2^{jn}\check{\varphi}(2^jx))$.  Observe that, if $|j-j'|\geq 2$, then supp ${\varphi_j}$ $\cap$ supp
${\varphi_{j'}} = \emptyset$.  We define ${\psi}_n \in S(\R^n)$ by the equality
\begin{equation*}
{\psi}_n (\xi) = 1 - \sum_{j\geq n} \varphi_j(\xi) 
\end{equation*}
for all $\xi\in\R^n$, and for $f\in S'(\R^n)$ we define the operator $S_nf$ by  
\begin{equation*}
S_n f = {\check{\psi}}_n \ast f.
\end{equation*}
In the following sections we will make frequent use of both the homogeneous and the inhomogeneous Littlewood-Paley operators.  For $f\in S'(\R^n)$ and $j\in\Z$, we define the homogeneous Littlewood-Paley operators ${\dot{\Delta}}_j$ by
\begin{equation*}
{\dot{\Delta}}_j f= {\check{\varphi}}_j \ast f,
\end{equation*}
and we define the inhomogeneous Littlewood-Paley operators by
\begin{align*}
    \Delta_j f = \left\{
        \begin{array}{ll}
            0,
                & j < -1, \\
            \check{\psi_0} \ast f,
                & j = -1, \\
            {\check{\varphi}}_j \ast f,
                & j > -1.
        \end{array}
        \right.
\end{align*}  
We remark that the operators $\Delta_j$ and $\dot{\Delta}_j$ coincide when $j\geq 0$, but differ when $j\leq -1$.

In the proof of the main theorem we use the paraproduct decomposition introduced by J.-M. Bony in \cite{Bony}.
We recall the definition of the paraproduct and remainder used in this decomposition. 
\begin{definition}\label{para}
Define the paraproduct of two functions $f$ and $g$ by
\begin{equation*}
T_fg = \sum_{\stackrel{i,j}{i\leq j-2}} \Delta_i f\Delta_j g = \sum_{j=1}^{\infty} S_{j-1}f\Delta_j g.
\end{equation*}
We use $R(f,g)$ to denote the remainder.  $R(f,g)$ is given by the following bilinear operator:  
\begin{equation*}
R(f,g)=\sum_{\stackrel{i,j}{|i-j|\leq 1}} \Delta_if\Delta_jg.
\end{equation*}
\end{definition}
\noindent Bony's decomposition then gives
\begin{equation*}
fg=T_fg + T_gf + R(f,g).
\end{equation*} 
We now define the homogeneous Besov spaces.
\begin{definition}\label{besovhomo}
Let $s\in\R$, $(p,q)\in[1,\infty]\times[1,\infty)$.  We define the homogeneous Besov space ${\dot{B}}^s_{p,q}(\R^n)$ to be the space of tempered distributions $f$ on $\R^n$ such that\\
\begin{equation*}
||f||_{{\dot{B}}^s_{p,q}}:={\left(\sum_{j\in\Z}2^{jqs}{||{\dot{\Delta}}_jf||}^q_{L^p}\right)}^{\frac{1}{q}} < \infty.
\end{equation*} 
When $q=\infty$, write
\begin{equation*}
||f||_{{\dot{B}}^s_{p,\infty}}:={\sup_{j\in\Z}{2^{js}{||{\dot{\Delta}}_jf||}_{L^p}}}.
\end{equation*} 
\end{definition}
\Obsolete{\begin{definition}\label{besov}
Let $s\in\R$, $(p,q)\in[1,\infty]\times[1,\infty)$.  We define the inhomogeneous Besov space $B^s_{p,q}(\R^n)$ to be the space of tempered distributions $f$ on $\R^n$ such that
\begin{equation*}
||f||_{B^s_{p,q}}:=||S_0f||_{L^p}+{\left(\sum_{j=0}^{\infty}2^{jqs}{||{\dot{\Delta}}_jf||}^q_{L^p}\right)}^{\frac{1}{q}} < \infty.
\end{equation*} 
When $q=\infty$, write
\begin{equation*}
||f||_{{B}^s_{p,\infty}}:=||S_0f||_{L^p}+{\sup_{j\geq 0}{2^{js}{||{\dot{\Delta}}_jf||}_{L^p}}}.
\end{equation*} 
\end{definition}}
We also define the inhomogeneous Zygmund spaces.
\begin{definition}
The inhomogeneous Zygmund space $C^s_{\ast}(\R^n)$ is the set of all tempered distributions $f$ on $\R^n$ such that 
\begin{equation*}
||f||_{{C}^s_{\ast}}:=\sup_{j\geq -1}2^{js}||{\Delta}_j f||_{L^{\infty}} < \infty.
\end{equation*} 
\end{definition}
It is well-known that $C^s_{\ast}(\R^n)$ coincides with the classical Holder space $C^s(\R^n)$ when $s$ is not an integer and $s>0$.

We will make frequent use of Bernstein's Lemma.  We refer the reader to \cite{Chemin1}, chapter 2, for a proof of the lemma.
\begin{lem}\label{bernstein}
(Bernstein's Lemma) Let $r_1$ and $r_2$ satisfy $0<r_1<r_2<\infty$, and let $p$ and $q$ satisfy $1\leq p \leq q \leq \infty$. There exists a positive constant $C$ such that for every integer $k$ , if $u$ belongs to $L^p(\R^n)$, and supp $\hat{u}\subset B(0,r_1\lambda)$, then 
\begin{equation}\label{bern1}
\sup_{|\alpha|=k} ||\partial^{\alpha}u||_{L^q} \leq C^k{\lambda}^{k+n(\frac{1}{p}-\frac{1}{q})}||u||_{L^p}.
\end{equation}
Furthermore, if supp $\hat{u}\subset C(0, r_1\lambda, r_2\lambda)$, then 
\begin{equation}\label{bern2}
C^{-k}{\lambda}^k||u||_{L^p} \leq \sup_{|\alpha|=k}||\partial^{\alpha}u||_{L^p} \leq C^{k}{\lambda}^k||u||_{L^p}.
\end{equation} 
\end{lem} 
As a result of Bernstein's Lemma, we have the following lemma regarding the homogeneous Besov spaces.
\begin{lem}\label{homoprop}
Let $u\in S'$, $s\in\R$, and $p,q\in[1,\infty]$.  For any $k\in\Z$  there exists a constant $C_k$ such that whenever $|\alpha|=k$,
\begin{equation*}
{C}^{-k}||\partial^{\alpha}u||_{{\dot{B}}^s_{p,q}} \leq ||u||_{{\dot{B}}^{s+k}_{p,q}} \leq C^k||\partial^{\alpha}u||_{{\dot{B}}^s_{p,q}}.
\end{equation*}
\end{lem}
We also make use of the following technical lemma.  We refer the reader to \cite{CK} for a detailed proof.
\begin{lemma}\label{CZhighfreq}
Let $v$ be a divergence-free vector field with vorticity $\omega$, and let $v$ and $\omega$ satisfy the relation $\nabla v=\nabla{\nabla}^{\perp}{\Delta}^{-1}\omega$.  Then there exists an absolute constant $C$ such that for all $j\in\Z$,
    \begin{align*}
        ||{\dot{\Delta}}_j \nabla v||_{L^{\infty}}
            \leq C||{\dot{\Delta}}_j \omega||_{L^{\infty}}.
    \end{align*}
\end{lemma}
\Obsolete{\begin{remark}
In what follows, we repeatedly use the equality $\nabla v=\nabla{\nabla}^{\perp}{\Delta}^{-1}\omega$, where $\omega=\omega(v)$.  This equality is justified for nondecaying velocity and vorticity by the following argument.  If we let $v'={\nabla}^{\perp}{\Delta}^{-1}\omega$, then we see that $v$ and $v'$ have the same divergence and vorticity, given by $\omega$.  Therefore, it follows by the equality
\begin{equation}
\Delta v_i = \partial_i \text{div } v + \sum_i \partial_i \omega(v)
\end{equation}
that $\Delta v = \Delta v'$.  Hence $v-v'=f$, where $f$ is a harmonic polynomial.  To see that $f$ is a constant, note that
\begin{equation*}
\begin{split}
&||v'||_{BMO} = ||{\nabla}^{\perp}{\Delta}^{-1}\omega||_{BMO}\\
&\qquad \leq ||{\nabla}^{\perp}{\Delta}^{-1}\nabla v||_{BMO} \leq C||v||_{L^{\infty}},\\
\end{split}
\end{equation*}
where we used the boundedness of Calderon-Zygmund operators from $L^{\infty}$ to $BMO$ to get the last inequality.  We conclude that $f=v'-v$ belongs to $BMO$, and since $BMO$ contains no polynomial of degree one or higher, $f$ is a constant.  This yields the desired equality,
\begin{equation*}
\nabla v=\nabla v'=\nabla{\nabla}^{\perp}{\Delta}^{-1}\omega.
\end{equation*}           
\end{remark}}
We will need a uniform bound in time on the $L^{\infty}$-norms of the vorticities corresponding to the solutions of ($NS$) and ($E$).  For fixed $\nu\geq 0$, we have that 
\begin{equation}\label{NSvortbound}
||\omega_{\nu}(t)||_{L^{\infty}} \leq ||\omega_{\nu}^0||_{L^{\infty}}
\end{equation}
for all $t\geq 0$.  One can prove this bound by applying the maximum principle to the vorticity formulations of ($NS$) and ($E$).  We refer the reader to Lemma 3.1 of \cite{ST} for a detailed proof.  
\section{Properties of Nondecaying solutions to the fluid equations}
In this section, we summarize what is known about nondecaying solutions to ($NS$) and ($E$).  We begin with the mild solutions to ($NS$) established in \cite{GIM}.  By a mild solution to ($NS$), we mean a solution $v_{\nu}$ of the integral equation
\begin{equation}\label{INT}
v_{\nu}(t,x)=e^{t\nu\Delta}v_{\nu}^0 - \int_0^t{e^{(t-s)\nu\Delta} \begin{bf} P \end{bf} (v_{\nu}\cdot\nabla v_{\nu})(s)}ds.
\end{equation}
In ($\ref{INT}$), $e^{\tau\nu\Delta}$ denotes convolution with the Gauss kernel; that is, for $f\in S'$, $e^{\tau\nu\Delta}f=G_{\tau\nu}\ast f$, where $G_{\tau\nu}(x)=\frac{1}{4\pi \tau\nu}\exp \{\frac{-{|x|}^2}{4\tau\nu }\}$.  Also, $\begin{bf} P \end{bf}$ denotes the Helmholtz projection operator with $ij$-component given by $\delta_{ij} + R_i R_j$, where $R_l=(-\Delta)^{-\frac{1}{2}}\partial_l$ is the Riesz operator.  In \cite{GIM}, Giga, Inui, and Matsui prove the following result regarding mild solutions in $\R^n$, $n\geq 2$.
\begin{theorem}\label{GMS}
Let $BUC$ denote the space of bounded, uniformly continuous functions, and assume ${v_{\nu}}^0$ belongs to $BUC(\R^n)$ for fixed $n\geq 2$.  There exists a $T_0>0$ and a unique solution to {\em ($\ref{INT}$)} in the space $C([0,T_0]; BUC(\R^n))$ with initial data ${v_{\nu}}^0$.  Moreover, if we assume div ${v_{\nu}}^0=0$, and if we define $p_{\nu}(t)=\sum_{i,j=1}^2R_iR_j v_{\nu i} v_{\nu j}(t)$ for each $t\in[0,T_0]$, then $v_{\nu}$ belongs to $C^{\infty}([0,T_0]\times \R^n)$ and solves ($NS$).
\end{theorem} 
\begin{remark}
In the above theorem, one can assume $v_{\nu}^0\in L^{\infty}(\R^n)$ and draw similar conclusions.  Indeed, with this weaker assumption the theorem is still true as long as one replaces $C([0,T_0]; BUC(\R^n))$ with $C_w([0,T_0]; BUC(\R^n))$, where $C_w$ denotes the space of weakly continuous functions.  For the main theorem of this paper, we assume that $\omega^0$ is bounded on $\R^2$, which, by Lemma \ref{C1}, implies that $v^0_{\nu}$ belongs to $C^{\alpha}(\R^2)$ for every $\alpha<1$.  Therefore, the statement of the theorem with $v_{\nu}^0\in BUC(\R^2)$ applies in our case.
\end{remark}  
\begin{remark}
Note that when the solution $v_{\nu}$ belongs to $L^r(\R^n)$ for $r<\infty$, the pressure can be determined from $v_{\nu}$ up to a constant, giving uniqueness of $(v_{\nu}, \nabla p_{\nu})$ without any assumptions on the pressure.  However, without a decay assumption on the velocity, this relation between velocity and pressure does not necessarily follow.  As a result, the authors are forced to place a restriction on the pressure in the statement of the theorem in order to establish uniqueness of $(v_{\nu}, \nabla p_{\nu})$.  The question of necessary assumptions on $p_{\nu}$ to ensure uniqueness of $(v_{\nu}, \nabla p_{\nu})$ is addressed by Kato in \cite{K}.   He shows that $(v_{\nu}, \nabla p_{\nu})$ can be uniquely determined when $p_{\nu}$ belongs to $L^1_{loc}([0,T); BMO(\R^n))$.  We refer the reader to \cite{K} for further details.
\end{remark}
In \cite{GMS}, Giga, Matsui, and Sawada show that when $n=2$, the solution to ($NS$) established in Theorem \ref{GMS} can be extended to a global-in-time smooth solution.  Moreover, in \cite{ST}, Sawada and Taniuchi show that if $v_{\nu}^0$ and $\omega_{\nu}^0$ belong to $L^{\infty}(\R^2)$, then the following exponential estimate holds:
\begin{equation}\label{NSunifbd}
||v_{\nu}(t)||_{L^{\infty}} \leq C||v_{\nu}^0||_{L^{\infty}}e^{Ct||\omega_{\nu}^0||_{L^{\infty}}}.
\end{equation} 
For ideal incompressible fluids, Serfati proves the following existence and uniqueness result in \cite{Serfati}.
\begin{theorem}\label{Serfati}
Let $v^0$ and $\omega^0$ belong to $L^{\infty}(\R^2)$, and let $c\in\R$.  For every $T>0$ there exists a unique solution $(v,p)$ to ($E$) in the space $L^{\infty}([0,T]; L^{\infty}(\R^2))\times L^{\infty}([0,T]; C(\R^2))$ with $\omega\in L^{\infty}([0,T]; L^{\infty}(\R^2))$, $p(0)=c$, and with $\frac {p(t,x)}{|x|}\rightarrow 0$ as $|x|\rightarrow \infty$.  
\end{theorem}
The unique solution to ($E$) given in Theorem \ref{Serfati} satisfies an integral representation analogous to that for mild solutions to ($NS$).  Specifically, the Serfati solution satisfies the equation
\begin{equation}\label{EINT}
v(t,x) = v(0,x)-\int_0^t{\begin{bf} P \end{bf} (v\cdot\nabla v)(s)}ds.
\end{equation} 
\Obsolete{To establish ($\ref{EINT}$), we use the following lemma.
\begin{lemma}\label{Easint}
Let $(v,p)$ be the unique solution to ($E$) which satisfies the conditions of Theorem \ref{Serfati}.  Then the pressure $p$ satisfies the equality 
\begin{equation*}
p(t) = \sum_{i,j=1}^2\Delta^{-1}\partial_i\partial_j (v^iv^j)(t) +C
\end{equation*}
for every $t\geq 0$, where $C$ is an absolute constant.
\end{lemma}
\begin{proof}
Let $p'= \sum_{i,j=1}^2\Delta^{-1}\partial_i\partial_j (v^iv^j)$.  Taking the divergence of ($E$), we see that 
\begin{equation*}
\Delta p = \sum_{i,j=1}^2\partial_i\partial_j (v^iv^j).
\end{equation*}
Therefore, $\Delta p=\Delta p'$, which implies that $p$ and $p'$ differ by a harmonic polynomial $f$.  To complete the proof, it suffices to show that $f$ is a constant.  

Since $v$ is bounded, it follows by properties of Calderon-Zygmund operators that $p'$ belongs to the space $BMO$ (see, for example, \cite{Stein}).  Moreover, by the properties of the pressure $p$ corresponding to the Serfati solution $v$, we have that $\frac{p(x)}{|x|}$ approaches $0$ as $|x|$ approaches infinity.  Assume, for contradiction, that $f$ is a polynomial of degree greater than or equal to one.  
Given $\epsilon>0$, for $|x|$ sufficiently large, $p'(x)$ satisfies
\begin{equation*}
\begin{split}
|p'(x)| &\geq |f(x)|-|p(x)|\\
&\geq |f(x)|-\epsilon|x|.
\end{split} 
\end{equation*}
It follows that for large $x$, $p'$ behaves like a polynomial of degree greater than or equal to one, which contradicts the membership of $p'$ to $BMO$.  Therefore $f$ must be a constant.  This completes the proof.           
\end{proof}  }
Serfati also establishes an estimate analogous to ($\ref{NSunifbd}$) for the Euler equations.  He proves the bound
\begin{equation}\label{Eunifbd}
||v(t)||_{L^{\infty}} \leq C||v^0||_{L^{\infty}}e^{Ct||\omega^0||_{L^{\infty}}}.
\end{equation}

Before we state the main theorem of the paper, we prove a result giving Holder regularity of solutions to ($NS$) and ($E$) with initial velocity and vorticity in $L^{\infty}(\R^2)$.  We prove that under these assumptions on the initial data, the corresponding solution to ($NS$) or ($E$) belongs to the Zygmund space $C^{1}_{\ast}$.  We prove the lemma only for the solution to ($E$).  The proof for the ($NS$) solution is identical.
\begin{lemma}\label{C1}
Let $v$ be the unique solution to ($E$) given by Theorem \ref{Serfati} with bounded initial vorticity and bounded initial velocity.  Then the following estimate holds:
\begin{equation*}
||v(t)||_{C^{1}_{\ast}} \leq C||v^0||_{L^{\infty}}e^{Ct||\omega^0||_{L^{\infty}}} + C||\omega^0||_{L^{\infty}}.
\end{equation*}
\end{lemma}
\begin{proof}
Write 
\begin{equation*}
||v(t)||_{C^{1}_{\ast}} \leq C||S_0 v(t)||_{L^{\infty}} + \sup_{j\geq 0}2^j||{\dot{\Delta}}_j v(t)||_{L^{\infty}}.
\end{equation*}
We first use Young's inequality to bound the low frequency term by $C||v(t)||_{L^{\infty}}$.  We then apply the bound given in ($\ref{Eunifbd}$).  For the high frequency terms, we apply Bernstein's Lemma and the uniform estimate on the $L^{\infty}$-norm of the vorticity given in ($\ref{NSvortbound}$) to bound the supremum by $C||\omega^0||_{L^{\infty}}$.  This completes the proof. 
\end{proof}
\section{Statement and Proof of the Main Result}
\noindent We are now prepared to state the main theorem.
\begin{theorem}\label{main}
Let $v_{\nu}$ be the unique solution to ($NS$) and $v$ the unique solution to ($E$), both with initial data $v^0$ and $\omega^0$ belonging to $L^{\infty}(\R^2)$, and with $p_{\nu}$ and $p$ satisfying the conditions of Theorems \ref{GMS} and \ref{Serfati}, respectively.  Then there exist constants $C$ and $C_1$, depending only on $||v^0||_{L^{\infty}}$ and $||\omega^0||_{L^{\infty}}$, such that the following estimate holds for any fixed $\alpha\in(0,1)$ and for $\nu$ sufficiently small:
\begin{equation}\label{thebigone}
||v_{\nu}-v||_{L^{\infty}([0,T];L^{\infty}(\R^2))} \leq C(T+1)e^{C_1T}{(\nu)}^{\frac{\alpha}{2}}\{e^{C(e^{C_1T}-1)}\}^{-\frac{1}{2}\log_2 {\nu}}.
\end{equation}
\end{theorem}
\begin{remark}
The constant $C_1$ in ($\ref{thebigone}$) is equal to $A||\omega^0||_{L^{\infty}}$, where $A$ is an absolute constant.  Therefore, one can conclude from Theorem \ref{main} that the vanishing viscosity limit holds on a time interval with length inversely proportional to the size of $||\omega^0||_{L^{\infty}}$.  Moreover, the smallness of $\nu$ required to conclude ($\ref{thebigone}$) depends on $T$ (and therefore on $||\omega^0||_{L^{\infty}}$).  Specifically, larger $T$ requires smaller $\nu$ for ($\ref{thebigone}$) to hold (see Remark \ref{killnu}). 
\end{remark}
\begin{proof}
Let $v$ be the unique solution to ($E$) with bounded initial velocity and vorticity.  In what follows, we let $v_n=S_nv$, and $\omega_n=S_n\omega(v)$.  We have the following inequality:
\begin{equation}\label{3terms}
\begin{split}
&||v_{\nu}-v||_{L^{\infty}([0,T];L^{\infty}(\R^2))} \leq ||S_{-n}(v_{\nu}-v)||_{L^{\infty}([0,T];L^{\infty}(\R^2))}\\
&\qquad+ ||(Id-S_{-n})(v_{\nu}-v_n)||_{L^{\infty}([0,T];L^{\infty}(\R^2))}\\
&\qquad+ ||(Id-S_{-n})(v_{n}-v)||_{L^{\infty}([0,T];L^{\infty}(\R^2))}.
\end{split}
\end{equation}  
We will estimate each of the three terms on the right hand side of the inequality in ($\ref{3terms}$).  We begin with the third term, since it is the easiest to handle.  We use the definition of $v_n$, Bernstein's Lemma, Lemma \ref{CZhighfreq}, and (\ref{NSvortbound}) to obtain the inequality 
\begin{equation}\label{term3}
||(Id-S_{-n})(v_{n}-v)||_Y \leq C2^{-n}||\omega^0||_{L^{\infty}}.
\end{equation} 
To bound the first term on the right hand side of ($\ref{3terms}$), we will use the integral representation of the $L^{\infty}(\R^2)$ solution of ($NS$) given in (\ref{INT}), as well as the integral representation of the Serfati solution to ($E$) given in ($\ref{EINT}$).
\noindent We prove the following proposition.
\begin{prop}\label{lowfreq}
Let $v_{\nu}$ and $v$ be solutions to ($NS$) and ($E$), respectively, satisfying the assumptions of Theorem \ref{main}.  Then for any $\alpha\in (0,1)$ and for any $\delta>0$,
\begin{equation*}
||S_{-n}(v_{\nu}-v)(t)||_{L^{\infty}} \leq Ce^{C_1t}\left\{2^{-n} + {\delta}^{\alpha}+ \exp\left(-\frac{\delta^2}{4\nu t}\right)\right\},
\end{equation*}
where $C$ and $C_1$ depend only on $||v^0||_{L^{\infty}}$ and $||\omega^0||_{L^{\infty}}$.
\end{prop}
\begin{proof}
To prove Proposition $\ref{lowfreq}$, we apply the operator $S_{-n}$ to ($\ref{INT}$) and ($\ref{EINT}$), we subtract the modification of ($\ref{EINT}$) from that of ($\ref{INT}$), and we take the $L^{\infty}$-norm to get 
\begin{equation*}
\begin{split}
&||S_{-n}(v_{\nu}-v)(t)||_{L^{\infty}} \leq ||S_{-n}(e^{t\nu\Delta}v^0-v^0)||_{L^{\infty}}\\
& +  \int_0^t{||S_{-n}e^{(t-s)\nu\Delta} \begin{bf} P \end{bf} (v_{\nu}\cdot\nabla v_{\nu})(s)||_{L^{\infty}}}ds + \int_0^t{||S_{-n} \begin{bf} P \end{bf} (v\cdot\nabla v)(s)||_{L^{\infty}}}ds.
\end{split}
\end{equation*}
We first estimate $||S_{-n}e^{(t-s)\nu\Delta} \begin{bf} P \end{bf} (v_{\nu}\cdot\nabla v_{\nu})(s)||_{L^{\infty}}$.  We follow an argument of Taniuchi in \cite{Taniuchi} which uses the boundedness of the Helmholtz projection operator on the Hardy space ${\mathcal H}^1(\R^2)$.  We have  
\begin{equation}\label{NSlowfreq}
\begin{split}
&||S_{-n}e^{(t-s)\nu\Delta} \begin{bf} P \end{bf} (v_{\nu}\cdot\nabla v_{\nu})(s)||_{L^{\infty}}\\
&\qquad \leq C2^{-n}||v_{\nu}(s)||^2_{L^{\infty}},
\end{split}
\end{equation}
after applying Holder's inequality and the series of inequalities given by $||\nabla \begin{bf}P \end{bf}\psi_{-n}||_{L^1(\R^2)} \leq ||\nabla \begin{bf}P \end{bf}\psi_{-n}||_{{\mathcal H}^1(\R^2)}\leq ||\nabla\psi_{-n}||_{{\mathcal H}^1(\R^2)}\leq C2^{-n}$.  Proofs of ($\ref{NSlowfreq}$) can also be found in \cite{ST} and \cite{ST1}.  

Similarly, we have
\begin{equation}\label{Elowfreq}
||S_{-n}\begin{bf}P\end{bf}(v\cdot \nabla v)(s)||_{L^{\infty}} \leq  C2^{-n}||v(s)||^2_{L^{\infty}}.  
\end{equation}
Using ($\ref{NSunifbd}$) and ($\ref{Eunifbd}$), we can bound $||v_{\nu}(s)||_{L^{\infty}}$ and $||v(s)||_{L^{\infty}}$ with a constant depending on $t$, $||v^0||_{L^{\infty}}$, and $||\omega^0||_{L^{\infty}}$.

It remains to estimate $||S_{-n}(e^{t\nu\Delta}v^0-v^0)||_{L^{\infty}}$.  To bound this difference, we use the following lemma.
\begin{lemma}\label{Gauss}
Let $u$ belong to the Holder space $C^{\alpha}(\R^2)$.  Then for any fixed $\delta>0$, the following estimate holds:
\begin{equation*}
||e^{t\nu\Delta}u-u||_{L^{\infty}} \leq C||u||_{L^{\infty}}\frac{1}{\exp( \frac{{\delta}^2}{4\nu t})} + {\delta}^{\alpha}||u||_{C^{\alpha}}.
\end{equation*}  
\end{lemma}
\begin{proof}
We fix $\delta>0$ and write 
\begin{equation}\label{1000}
\begin{split}
&|e^{t\nu\Delta}u(x)-u(x)| \leq \left |\int_{|y|\leq\delta}{G_{t\nu}(y)u(x-y)}dy -u(x)\right |\\
&\qquad\qquad + \left|\int_{|y|>\delta}{G_{t\nu}(y)u(x-y)}dy\right|.
\end{split}
\end{equation}
After integrating, we can bound the second term on the right hand side of ($\ref{1000}$) by
\begin{equation}\label{1001}
C||u||_{L^{\infty}}\exp\left(-\frac{\delta^2}{4\nu t}\right). 
\end{equation}
To estimate the first term on the right hand side of ($\ref{1000}$), we use the property that the integral over $\R^2$ of the Gauss kernel is equal to one to write 
\begin{equation*}
\begin{split}
&\left|\int_{|y|\leq\delta}{ G_{t\nu}(y)u(x-y)}dy -u(x)\right|\\
& \leq \left|\int_{|y|\leq\delta}{ G_{t\nu}(y)u(x-y)}dy -\int_{\R^2}{G_{t\nu}(y)u(x)}dy\right|\\
&\leq \int_{|y|\leq \delta}{G_{t\nu}(y)|u(x-y)-u(x)|}dy + \int_{|y|> \delta}{G_{t\nu}(y)|u(x)|}dy\\
& \leq \sup_{|y|\leq\delta} |u(x-y)-u(x)| + C||u||_{L^{\infty}}\exp\left(-\frac{\delta^2}{4\nu t}\right),
\end{split}
\end{equation*}
where we utilized the bound given in ($\ref{1001}$) on the second term.  We now use the membership of $u$ to $C^{\alpha}(\R^2)$ for every $\alpha<1$ to bound the first term by ${\delta}^{\alpha}||u||_{C^{\alpha}}$.  This completes the proof.
\end{proof}
We apply Lemma \ref{Gauss} with $u=v^0$, and we combine the resulting estimate with $(\ref{NSlowfreq})$, $(\ref{Elowfreq})$, and Lemma \ref{C1}  to complete the proof of Proposition \ref{lowfreq}.  
\end{proof}
\begin{remark}\label{killnu}
If we let $\nu=2^{-2n}$ and $\delta=2^{-n\alpha}$, then, since $\alpha\in(0,1)$ is arbitrary, the estimate in Proposition \ref{lowfreq} reduces to
\begin{equation}\label{23}
||S_{-n}(v_{\nu}-v)(t)||_{\infty} \leq Ce^{C_1t}2^{-n\alpha}
\end{equation}
for $n\geq N$, with $N$ sufficiently large.  We remark here that, because of the appearance of $t$ on the right hand side of the inequality in Lemma \ref{Gauss}, the size of $N$ necessary to make ($\ref{23}$) hold depends on $t$.  In particular, larger $t$ requires larger $N$.  In Theorem \ref{main}, we work on a finite time interval $[0,T]$.  Therefore, we can choose $N$ large enough so that ($\ref{23}$) holds for all $t\in[0,T]$.
\end{remark}
It remains to bound the second term on the right hand side of ($\ref{3terms}$), given by $||(Id-S_{-n})(v_{\nu}-v_n)(t)||_{L^{\infty}}$.  We prove the following estimate.
\begin{prop}\label{technical}
Let $v_{\nu}$ and $v$ be solutions to ($NS$) and ($E$), respectively, satisfying the properties of Theorem \ref{main}.  Then there exist constants $C$ and $C_1$, depending only on the initial data, such that the following estimate holds for any fixed $\alpha\in(0,1)$ and for sufficiently large $n$:
\begin{equation}\label{zozozo}
||(Id-S_{-n})(v_{\nu}-v_n)(t)||_{L^{\infty}} \leq C(t+1)e^{C_1t}2^{-n\alpha}e^{Cn(e^{C_1t}-1)}.
\end{equation}
\end{prop}
\begin{remark}
In the proof of Proposition \ref{technical}, we let $\nu=2^{-2n}$ as in Remark \ref{killnu}.  Therefore, the dependence of the right hand side of ($\ref{zozozo}$) on $\nu$ is hidden in its dependence on $n$. 
\end{remark}
\begin{proof}
We begin with the series of inequalities
\begin{equation}\label{zoey}
\begin{split}
&||(Id-S_{-n})(v_{\nu}-v_n)(t)||_{L^{\infty}} \leq ||(Id-S_{n})(v_{\nu}-v_n)(t)||_{L^{\infty}}\\
&\qquad\qquad +||(S_n-S_{-n})(v_{\nu}-v_n)(t)||_{L^{\infty}}\\
&\qquad\qquad\leq C2^{-n}||\omega^0||_{L^{\infty}} + Cn||(v_{\nu}-v_n)(t)||_{{\dot{B}}^0_{\infty,\infty}}.\\
\end{split}
\end{equation}
The second inequality follows from an application of Bernstein's Lemma, Lemma \ref{CZhighfreq}, and the uniform bound on the vorticity given in ($\ref{NSvortbound}$).  

To bound $||(v_{\nu}-v_n)(t)||_{{\dot{B}}^0_{\infty,\infty}}$, we need the following lemma, whose proof we postpone until the next section.
\begin{lem}\label{key}
Let $v_{\nu}$ and $v$ be solutions to ($NS$) and ($E$), respectively, satisfying the properties of Theorem \ref{main}.  Then there exist constants $C$ and $C_1$, depending only on the initial data, such that the following estimate holds for any fixed $\alpha\in(0,1)$:
\begin{equation*}
\begin{split}
&||(v_{\nu}-v_n)(t)||_{{\dot{B}}^{0}_{\infty,\infty}} \leq C(t+1)e^{C_1t}2^{-n\alpha}\\
&\qquad\qquad + \int_0^t{Cne^{C_1s}||(v_{\nu}-v_n)(s)||_{{\dot{B}}^0_{\infty,\infty}}}ds.
\end{split}
\end{equation*}
\end{lem}
Assuming this lemma holds, we can apply Gronwall's Lemma and integrate in time to conclude that 
\begin{equation*}
||(v_{\nu}-v_n)(t)||_{{\dot{B}}^{0}_{\infty,\infty}} \leq C(t+1)e^{C_1t}2^{-n\alpha}e^{Cn(e^{C_1t}-1)}.
\end{equation*}
Since $\alpha\in(0,1)$ is arbitrary, we can write
\begin{equation*}
n||(v_{\nu}-v_n)(t)||_{{\dot{B}}^{0}_{\infty,\infty}} \leq C(t+1)e^{C_1t}2^{-n\alpha}e^{Cn(e^{C_1t}-1)}
\end{equation*}
for sufficiently large $n$.  Combining this estimate with the estimate given in ($\ref{zoey}$) yields Proposition \ref{technical}.  
\end{proof}
To complete the proof of Theorem \ref{main}, we combine ($\ref{3terms}$), Remark \ref{killnu}, Proposition \ref{technical}, and ($\ref{term3}$) to get the following estimate for large $n$:
\begin{equation*}
||(v_{\nu}-v)(t)||_{L^{\infty}} \leq C(t+1)e^{C_1t}2^{-n\alpha}e^{Cn(e^{C_1t}-1)}. 
\end{equation*}
Using the equality $n=-\frac{1}{2}\log_2 {\nu}$, we obtain ($\ref{thebigone}$).  This completes the proof of Theorem \ref{main}.  We devote the next section to the proof of Lemma \ref{key}.  
\end{proof}
\section{Proof of Lemma \ref{key}}
We begin with some notation.  For the proof of Lemma \ref{key}, we let ${\bar{\omega}}_n=\omega_{\nu}-\omega_n$ and ${\bar{v}}_n=v_{\nu}-v_n$.

To prove the lemma, we localize the frequencies of the vorticity formulations of ($E$) and ($NS$), and we consider the difference of the two resulting equations.  After localizing the frequency of the vorticity formulation of ($E$), we see that ${\dot{\Delta}}_j\omega_n$ satisfies the following equation: 
\begin{equation}\label{1}
\partial_t{\dot{\Delta}}_j\omega_n + v_n\cdot\nabla{\dot{\Delta}}_j\omega_n + [{\dot{\Delta}}_j, v_n\cdot\nabla]\omega_n= \nabla\cdot {\dot{\Delta}}_j{\tau}_n(v,\omega),
\end{equation} 
where
\begin{equation*}
\tau_n(v,\omega)= r_n(v,\omega) - (v-v_n)(\omega-\omega_n)
\end{equation*}
and
\begin{equation*}
r_n(v,\omega) = \int{\check{\psi}(y)(v(x-2^{-n}y)-v(x))(\omega(x-2^{-n}y)-\omega(x))}dy.
\end{equation*}  
This equation is utilized by Constantin and Wu in \cite{CW} and by Constantin, E, and Titi in a proof of Onsager's conjecture in \cite{CET}.    

If $v_{\nu}$ is a solution to ($NS$), we can localize the frequency of the vorticity formulation of ($NS$) to see that ${\dot{\Delta}}_j\omega_{\nu}$ satisfies
\begin{equation}\label{NS}
\partial_t {\dot{\Delta}}_j\omega_{\nu} + v_{\nu}\cdot\nabla{\dot{\Delta}}_j\omega_{\nu} + [{\dot{\Delta}}_j, v_{\nu}\cdot\nabla]\omega_{\nu} = \nu\Delta{\dot{\Delta}}_j\omega_{\nu}.
\end{equation} 
We subtract ($\ref{1}$) from ($\ref{NS}$).  This yields
\Obsolete{\begin{equation*}
\begin{split}
&\partial_t{\dot{\Delta}}_j{\bar{\omega}}_n + v_n\cdot\nabla{\dot{\Delta}}_j{\bar{\omega}}_n + {\dot{\Delta}}_j({\bar{v}}_n\cdot\nabla {\omega}_{\nu}) -\nu\Delta {\dot{\Delta}}_j {\bar{\omega}}_n -\nu\Delta{{\dot{\Delta}}_j\omega_n} \\
&\qquad + \nabla\cdot {\dot{\Delta}}_j{\tau}_n(v,\omega)+ [{\dot{\Delta}}_j, v_n\cdot\nabla]{\bar{\omega}}_n=0,
\end{split}
\end{equation*}
which gives}
\begin{equation}\label{zozo}
\begin{split}
&\partial_t{\dot{\Delta}}_j{\bar{\omega}}_n + v_n\cdot\nabla{\dot{\Delta}}_j{\bar{\omega}}_n -\nu\Delta {\dot{\Delta}}_j{\bar{\omega}}_n = -{\dot{\Delta}}_j({\bar{v}}_n\cdot\nabla {\omega}_{\nu}) \\
&\qquad + \nu\Delta{{\dot{\Delta}}_j\omega_n} - \nabla\cdot {\dot{\Delta}}_j{\tau}_n(v,\omega) - [{\dot{\Delta}}_j, v_n\cdot\nabla]{\bar{\omega}}_n.
\end{split}
\end{equation}
We observe that $v_n$ is a divergence-free Lipschitz vector field and apply the following lemma, which is proved in \cite{H}. 
\begin{lem}\label{Hmidi}
Let $p\in[1,\infty]$, and let $u$ be a divergence-free vector field belonging to $L^1_{loc}(\R^+; Lip(\R^d))$.  Moreover, assume the function $f$ belongs to $L^1_{loc}(\R^+; L^p(\R^d))$ and the function $a^0$ belongs to $L^p(\R^d)$.  Then any solution $a$ to the problem
\begin{align*}
    \begin{matrix}
         & \left\{
            \begin{matrix}
                \partial_t a + u \cdot \nabla a - \nu \Delta
                 a = f, \\
                a|_{t = 0} = a^0
            \end{matrix}
            \right.
    \end{matrix}
\end{align*}
satisfies the following estimate:
\begin{equation*}
||a(t)||_{L^p}\leq ||a^0||_{L^p} + \int_0^t{||f(s)||_{L^p}}ds.
\end{equation*} 
\end{lem} 
An application of Lemma \ref{Hmidi} to ($\ref{zozo}$) yields
\begin{equation*}
\begin{split}
&||{\dot{\Delta}}_j{\bar{\omega}}_n(t)||_{L^{\infty}} \leq ||{\dot{\Delta}}_j{{\bar{\omega}}_n}^0||_{L^{\infty}} + C\int_0^t  {||(-{\dot{\Delta}}_j({\bar{v}}_n\cdot\nabla {\omega}_{\nu}) + \nu\Delta{{\dot{\Delta}}_j\omega_n} }\\
 &\qquad -{\nabla\cdot {\dot{\Delta}}_j{\tau}_n(v,\omega) - [{\dot{\Delta}}_j, v_n\cdot\nabla]{\bar{\omega}}_n)(s)||_{L^{\infty}}}ds.
\end{split}
\end{equation*}   
Multiplying through by $2^{-j}$ and taking the supremum over $j\in\Z$, we obtain the inequality
\begin{equation}\label{terms}
\begin{split}
&||{\bar{\omega}}_n(t)||_{{\dot{B}}^{-1}_{\infty,\infty}} \leq ||{{\bar{\omega}}_n}^0||_{{\dot{B}}^{-1}_{\infty,\infty}}+ C\int_0^t\{{||{\bar{v}}_n{\omega}_{\nu}(s)||_{{\dot{B}}^{0}_{\infty,\infty}} + ||\nu\nabla\omega_n(s)||_{{\dot{B}}^{0}_{\infty,\infty}}}\\
&\qquad { + ||{\tau}_n(v,\omega)(s)||_{{\dot{B}}^{0}_{\infty,\infty}} + \sup_{j\in\Z}2^{-j}||[{\dot{\Delta}}_j, v_n\cdot\nabla]{\bar{\omega}}_n(s)||_{L^{\infty}}}\}ds,
\end{split}
\end{equation}
where we repeatedly used the divergence-free condition on $v$ and Lemma \ref{homoprop}. 

To complete the proof of Lemma \ref{key}, we estimate each term on the right hand side of ($\ref{terms}$).  We begin by estimating $||{{\bar{\omega}}_n}^0||_{{\dot{B}}^{-1}_{\infty,\infty}}$.  Using the definition of ${{\bar{\omega}}_n}$ as well as the definition of the Besov space ${\dot{B}}^{-1}_{\infty,\infty}$, we see that
\begin{equation*}
\begin{split}
&||{{\bar{\omega}}_n}^0||_{{\dot{B}}^{-1}_{\infty,\infty}}
\leq {\sup_{j\geq n}2^{-j}||{\dot{\Delta}}_j(\omega^0-S_n\omega^0)||_{L^{\infty}}}  
\leq C2^{-n}||\omega^0||_{L^{\infty}}.
\end{split}
\end{equation*}
To bound $||{\bar{v}}_n{\omega}_{\nu}(s)||_{{\dot{B}}^{0}_{\infty,\infty}}$,  we observe that $L^{\infty}$ is continuously embedded in ${\dot{B}}^0_{\infty,\infty}$ and that the $L^{\infty}$-norm of vorticity is uniformly bounded in time by ($\ref{NSvortbound}$) to write 
\begin{equation}\label{badterm}
\begin{split}
&||{\bar{v}}_n{\omega}_{\nu}(s)||_{{\dot{B}}^0_{\infty,\infty}} \leq C||{\bar{v}}_n(s)||_{L^{\infty}}||{\omega}^0||_{L^{\infty}}.
\end{split}
\end{equation}
To estimate the $L^{\infty}$-norm of ${\bar{v}}_n(s)$, we use Remark \ref{killnu}, Bernstein's Lemma, ($\ref{NSvortbound}$), and the definition of the ${\dot {B}}^0_{\infty,\infty}$-norm to write
\begin{equation}\label{log}
\begin{split}
&||{\bar{v}}_n(s)||_{L^{\infty}} 
\leq ||S_{-n}(v_{\nu}-v)(s)||_{L^{\infty}}+||S_{-n}(v-v_n)(s)||_{L^{\infty}}\\
&\qquad +||(S_n-S_{-n})(v_{\nu}-v_n)(s)||_{L^{\infty}}+||(Id-S_n)(v_{\nu}-v_n)(s)||_{L^{\infty}}\\ 
&\qquad\leq Ce^{C_1s}2^{-n\alpha}+Cn||(v_{\nu}-v_n)(s)||_{{\dot B}^0_{\infty,\infty}}\\ 
\end{split}
\end{equation}
for fixed $\alpha\in (0,1)$.
Combining ($\ref{badterm}$) and ($\ref{log}$) gives 
\begin{equation}\label{Besov}
\begin{split}
&||{\bar{v}}_n {\omega}_{\nu}(s)||_{{\dot{B}}^{0}_{\infty,\infty}} \leq Ce^{C_1s}2^{-n\alpha}+Cn||(v_{\nu}-v_n)(s)||_{{\dot B}^0_{\infty,\infty}}.\\
\end{split}
\end{equation}
To bound $\nu||\nabla{\omega_n}(s)||_{{\dot{B}}^{0}_{\infty,\infty}}$, we again use Bernstein's Lemma, the definition of $\omega_n$, and ($\ref{NSvortbound}$) to conclude that
\begin{equation*}
\begin{split}
&\nu||\nabla{\omega_n}(s)||_{{\dot{B}}^{0}_{\infty,\infty}} 
\leq C\nu 2^{n} ||{\omega}^0||_{L^{\infty}}.
\end{split}
\end{equation*} 
If we let $\nu=2^{-2n}$ as in Remark \ref{killnu}, we obtain the inequality
\begin{equation*}
\nu||\nabla{\omega_n}(s)||_{{\dot{B}}^{0}_{\infty,\infty}}\leq C2^{-n} ||{\omega}^0||_{L^{\infty}}.
\end{equation*}  
Finally, we estimate $||{\tau}_n(v,\omega)(s)||_{{\dot{B}}^{0}_{\infty,\infty}}$.  We begin by estimating $||(v-v_n)(\omega-\omega_n)(s)||_{{\dot{B}}^{0}_{\infty,\infty}}$.  We again use the embedding $L^{\infty}\hookrightarrow {\dot{B}}^0_{\infty,\infty}$, Bernstein's Lemma, and ($\ref{NSvortbound}$) to write
\begin{equation}\label{nonterm1}
\begin{split}
&||(v-v_n)(\omega-\omega_n)(s)||_{{\dot{B}}^{0}_{\infty,\infty}}
\leq ||(v-v_n)(\omega-\omega_n)(s)||_{L^{\infty}}\\
&\qquad\qquad\leq C||\omega^0||_{L^{\infty}}2^{-n}||\omega^0||_{L^{\infty}}.
\end{split}
\end{equation}
In order to bound $||r_n(v,\omega)(s)||_{{\dot{B}}^{0}_{\infty,\infty}}$, we use the membership of $v$ to $C^{\alpha}(\R^2)$ for any $\alpha\in(0,1)$ to write
\begin{equation}\label{useholder}
|v(s,x-2^{-n}y)-v(s,x)| \leq C2^{-n\alpha}|y|^{\alpha}||v(s)||_{C^{\alpha}}.
\end{equation}
Since ${|y|}^{\alpha}|\check{\psi}(y)|$ is integrable, we can apply ($\ref{useholder}$) and Holder's inequality to conclude that
\begin{equation}\label{nonterm2}
\begin{split}
&||r_n(v,\omega)(s)||_{{\dot{B}}^{0}_{\infty,\infty}} 
\leq C2^{-n\alpha}||\omega^0||_{L^{\infty}}||v(s)||_{C^{\alpha}}\leq Ce^{C_1s}2^{-n\alpha},
\end{split}
\end{equation}  
where we used Lemma \ref{C1} to get the last inequality.  Here the constants $C$ and $C_1$ depend only on $||v^0||_{L^{\infty}}$ and $||\omega^0||_{L^{\infty}}$.  Combining ($\ref{nonterm1}$) and ($\ref{nonterm2}$) yields
\begin{equation*}
\begin{split}
||{\tau}_n(v,\omega)(s)||_{{\dot{B}}^{0}_{\infty,\infty}} & \leq Ce^{C_1s}2^{-n\alpha}.
\end{split}
\end{equation*}

It remains to bound the commutator term.  We use the following lemma, which we prove in the appendix.
\begin{lemma}\label{comm}
Let $v$ be a solution to ($E$) with vorticity $\omega=\omega(v)$, and assume $v$ and $\omega$ satisfy the conditions of Theorem \ref{main}.  Then the following commutator estimate holds:\\
\begin{equation*}
{\sup_{j\in\Z}2^{-j}||[{\dot{\Delta}}_j, v_n\cdot\nabla]{\bar{\omega}}_n(s)||_{L^{\infty}}}\leq Ce^{C_1s}2^{-n\alpha}+Ce^{C_1s}n||{\bar{\omega}}_n(s)||_{{\dot B}^{-1}_{\infty,\infty}},
\end{equation*}
where $C$ and $C_1$ depend only on $||v^0||_{L^{\infty}}$ and $||\omega^0||_{L^{\infty}}$.
\end{lemma}
In order to apply Gronwall's Lemma, we bound $||{\bar{\omega}}_n(s)||_{{\dot{B}}^{-1}_{\infty,\infty}}$ with $||{\bar{v}}_n(s)||_{{\dot{B}}^{0}_{\infty,\infty}}$, 
which yields
\begin{equation*}
\begin{split}
&{\sup_{j\in\Z}2^{-j}||[{\dot{\Delta}}_j, v_n\cdot\nabla]{\bar{\omega}}_n(s)||_{L^{\infty}}} \leq Ce^{C_1s}2^{-n\alpha}+Cne^{C_1s}||{\bar{v}}_n(s)||_{{\dot{B}}^{0}_{\infty,\infty}}.\\
\end{split}
\end{equation*}
Combining all of the above estimates gives, for fixed $\alpha\in(0,1)$,
\begin{equation}\label{3}
\begin{split}
&||{\bar{v}}_n(t)||_{{\dot{B}}^{0}_{\infty,\infty}}\leq C||{\bar{\omega}}_n(t)||_{{\dot{B}}^{-1}_{\infty,\infty}}\\
&\leq Ce^{C_1t}2^{-n\alpha} + \int_0^t{Cne^{C_1s}||{\bar{v}}_n(s)||_{{\dot{B}}^0_{\infty,\infty}}}ds.
\end{split}
\end{equation}
This completes the proof of Lemma \ref{key}.
\Obsolete{We can now apply Gronwall's Lemma to ($\ref{3}$) and integrate in time to conclude that 
\begin{equation*}
||{\bar{v}}_n(t)||_{{\dot{B}}^{0}_{\infty,\infty}} \leq Ce^{C_1t}2^{-n\alpha}e^{n(e^{C_1t}-1)}.
\end{equation*}
Since $\alpha$ is an arbitrary number in $(0,1)$, we can write
\begin{equation*}
n||{\bar{v}}_n(t)||_{{\dot{B}}^{0}_{\infty,\infty}} \leq Ce^{C_1t}2^{-n\alpha}e^{n(e^{C_1t}-1)}
\end{equation*}
for sufficiently large $n$.  Combining this estimate with the estimate given in ($\ref{zoey}$) completes the proof of Proposition \ref{technical}. 
To complete the proof of Theorem \ref{main}, we combine ($\ref{3terms}$), Proposition \ref{lowfreq}, Proposition \ref{technical}, and ($\ref{term3}$) to get the following estimate for large $n$:
\begin{equation*}
||(v_{\nu}-v)(t)||_{L^{\infty}} \leq Ce^{C_1t}2^{-n\alpha}e^{n(e^{C_1t}-1)}. 
\end{equation*}
Passing to the limit as $n$ approaches infinity, we get the desired result for sufficiently small $t$.  This completes the proof of Theorem \ref{main}.}
\section{Appendix}
\Obsolete{We now prove Lemma \ref{Gauss}.  We fix $\delta>0$ and write 
\begin{equation}\label{1000}
\begin{split}
&|e^{t\nu\Delta}u(x)-u(x)| \leq |\int_{|y|\leq\delta}{G_{t\nu}(y)u(x-y)}dy -u(x)|\\
&\qquad\qquad + |\int_{|y|>\delta}{G_{t\nu}(y)u(x-y)}dy|.
\end{split}
\end{equation}
After integrating by parts, we can bound the second term on the right hand side of ($\ref{1000}$) by
\begin{equation}\label{1001}
C||u||_{L^{\infty}}\exp\left(-\frac{\delta^2}{4\nu t}\right). 
\end{equation}
To estimate the first term on the right hand side of ($\ref{1000}$), we use the property that the integral over $\R^2$ of the Gauss kernel is equal to one to write 
\begin{equation*}
\begin{split}
&|\int_{|y|\leq\delta}{ G_{t\nu}(y)u(x-y)}dy -u(x)|\\
& \leq |\int_{|y|\leq\delta}{ G_{t\nu}(y)u(x-y)}dy -\int_{\R^2}{G_{t\nu}(y)u(x)}dy|\\
&\leq \int_{|y|\leq \delta}{G_{t\nu}(y)|u(x-y)-u(x)|}dy + \int_{|y|> \delta}{G_{t\nu}(y)|u(x)|}dy\\
& \leq \sup_{x\in\R^2}\sup_{|y|\leq\delta} |u(x-y)-u(x)| + C||u||_{L^{\infty}}\exp\left(-\frac{\delta^2}{4\nu t}\right),
\end{split}
\end{equation*}
where we utilized the bound given in ($\ref{1001}$) on the second term.  We now use the membership of $u$ to $C^{\alpha}(\R^2)$ for every $\alpha<1$ to bound the first term by ${\delta}^{\alpha}||u||_{C^{\alpha}}$.  This completes the proof.}
We now prove Lemma \ref{comm}.  We need to show that 
\begin{equation*}
{\sup_{j\in\Z}2^{-j}||[{\dot{\Delta}}_j, v_n\cdot\nabla]{\bar{\omega}}_n(s)||_{L^{\infty}}}\leq Ce^{C_1s}2^{-n\alpha}+Ce^{C_1s}n||{\bar{\omega}}_n(s)||_{{\dot B}^{-1}_{\infty,\infty}}.
\end{equation*}
We first use Bony's paraproduct decomposition to write
\begin{equation}\label{Bony}
\begin{split}
&[{\dot{\Delta}}_j, v_n\cdot\nabla]{\bar{\omega}}_n = \sum_{m=1}^2 [{\dot{\Delta}}_j, T_{v_n^m}\partial_m]{\bar{\omega}}_n\\
&\qquad + [{\dot{\Delta}}_j, T_{\partial_m\cdot}v_n^m]{\bar{\omega}}_n + [{\dot{\Delta}}_j, \partial_mR(v_n^m, \cdot)]{\bar{\omega}}_n.
\end{split}
\end{equation} 
To bound the $L^{\infty}$-norm of the first term on the right hand side of ($\ref{Bony}$), we consider the cases $j<0$ and $j\geq0$ separately.  For $j\geq0$, we use the definition of the paraproduct and properties of the partition of unity to establish the following equality:
\begin{equation}\label{C67}
[{\dot{\Delta}}_j, T_{v_n}\partial_m] = \sum_{j'=\max\{1,j-4\}}^{j+4} [{\dot{\Delta}}_j, S_{j'-1}(v_n)] \Delta_{j'}\partial_m.
\end{equation}
We then express the operator ${\dot{\Delta}}_j$ as a convolution with ${\check{\varphi}}_j$, write out the commutator on the right hand side of ($\ref{C67}$), and change variables.  This yields
\begin{equation}\label{stupidterm1}
\begin{split}
&||[{\dot{\Delta}}_j, T_{v_n}\partial_m]{\bar{\omega}}_n||_{L^{\infty}} \leq \sum_{j':|j-j'|\leq 4} ||\int {{\check{\varphi}}(y)(S_{j'-1}v_n(x-2^{-j}y)}\\
&\qquad {-S_{j'-1}v_n(x))\Delta_{j'}\partial_m{\bar{\omega}}_n(x-2^{-j}y)}dy||_{L^{\infty}}\\
&\qquad\leq \sum_{j':|j-j'|\leq 4} 2^{j'-j} ||S_{j'-1}\nabla v_n||_{L^{\infty}}|| \Delta_{j'}{\bar{\omega}}_n||_{L^{\infty}}\int {|\check{\varphi}(y)||y|} dy, \\
\end{split}
\end{equation}
where we used Bernstein's Lemma to get the last inequality.  To complete the argument for this term, we must estimate the growth of $||\nabla v_n||_{L^{\infty}}$ with $n$.  We break $\nabla v_n$ into low and high frequencies and recall the definition of $v_n$ to write
\begin{equation}\label{gradbound1}
||\nabla v_n||_{L^{\infty}} \leq ||\Delta_{-1}\nabla v_n||_{L^{\infty}} + \sum_{k=0}^n||\Delta_k\nabla v_n||_{L^{\infty}}.\\
\end{equation}
For the low frequency term, we bound $||\Delta_{-1}\nabla v_n||_{L^{\infty}}$ with $||v||_{L^{\infty}}$ and apply ($\ref{Eunifbd}$).  For the high frequencies, we apply Lemma \ref{CZhighfreq} and ($\ref{NSvortbound}$).  This yields
\begin{equation}\label{gradbound}
||\nabla v_n(s)||_{L^{\infty}} \leq Cne^{C_1s},
\end{equation}
where $C$ and $C_1$ depend only on $||v^0||_{L^{\infty}}$ and $||\omega^0||_{L^{\infty}}$.  Plugging this bound into ($\ref{stupidterm1}$), multiplying ($\ref{stupidterm1}$) by $2^{-j}$, and taking the supremum over $j\geq0$ gives
\begin{equation*}
\sup_{j\geq0}2^{-j} ||[{\dot{\Delta}}_j, T_{v_n}\partial_m]{\bar{\omega}}_n(s)||_{L^{\infty}} \leq Cne^{C_1s}||{\bar{\omega}}_n(s)||_{{\dot{B}}^{-1}_{\infty,\infty}}.
\end{equation*}
For the case $j<0$, we apply a different strategy.  We first reintroduce the sum over $m$ and expand the commutator. We then use the properties of our partition of unity, the assumption that $j<0$, and the divergence-free assumption on $v$ to establish the following series of inequalities:
\begin{equation}\label{J23}
\begin{split}
&\sum_{m=1}^2 ||[{\dot{\Delta}}_j, T_{v_n}\partial_m] {\bar{\omega}}_n||_{L^{\infty}} \leq \sum_{m=1}^2 \sum_{j'=1}^2 ||\partial_m\dot{\Delta}_j(S_{j'-1}v_n\Delta_{j'}{\bar{\omega}}_n)||_{L^{\infty}}\\
&\qquad\qquad + \sum_{m=1}^2 \sum_{j'=1}^2||S_{j'-1}v_n\Delta_{j'}{\dot{\Delta}}_j\partial_m{\bar{\omega}}_n||_{L^{\infty}} \\
&\qquad\qquad\leq C\sum_{j'=1}^2 2^j ||v||_{L^{\infty}}||\Delta_{j'}{\bar{\omega}}_n||_{L^{\infty}},\\
\end{split}
\end{equation}
where we applied Bernstein's Lemma to get the factor of $2^j$ in the last inequality.  We bound $||\Delta_{j'}{\bar{\omega}}_n||_{L^{\infty}}$ with $2^{j'}||\Delta_{j'}{\bar{v}}_n||_{L^{\infty}}$, again by Bernstein's Lemma, we multiply ($\ref{J23}$) by $2^{-j}$, and we take the supremum over $j<0$.  This gives
\begin{equation*}
\sup_{j<0} 2^{-j}\sum_{m=1}^2 ||[{\dot{\Delta}}_j, T_{v_n}\partial_m]{\bar{\omega}}_n ||_{L^{\infty}} \leq C||v||_{L^{\infty}}||{\bar{v}}_n||_{L^{\infty}}.
\end{equation*}
We now bound $||v||_{L^{\infty}}$ using ($\ref{Eunifbd}$), and we bound $||{\bar{v}}_n||_{L^{\infty}}$ as in ($\ref{log}$).  We obtain the desired estimate:
\begin{equation*}
\sup_{j<0}\sum_{m=1}^2 ||[{\dot{\Delta}}_j, T_{v_n}\partial_m] {\bar{\omega}}_n(s)||_{L^{\infty}} \leq Ce^{C_1s}2^{-n\alpha}+Cne^{C_1s}||{\bar{\omega}}_n(s)||_{{\dot{B}}^{-1}_{\infty,\infty}}.
\end{equation*}
We now estimate the $L^{\infty}$-norm of $[{\dot{\Delta}}_j, T_{\partial_m\cdot}v_n]{\bar{\omega}}_n$.  We write out the commutator and estimate the $L^{\infty}$-norm of ${\dot{\Delta}}_j(T_{\partial_m {\bar{\omega}}_n}v_n)$ and $T_{\partial_m{\dot{\Delta}}_j {\bar{\omega}}_n}v_n$ separately.  By the definition of the paraproduct and by properties of our partition of unity, we have
\begin{equation}\label{comm2}
\begin{split}
&||T_{\partial_m{\dot{\Delta}}_j {\bar{\omega}}_n}v_n||_{L^{\infty}} = ||\sum_{l\geq 1}S_{l-1}\partial_m{\dot{\Delta}}_j{\bar{\omega}}_n\Delta_l v_n||_{L^{\infty}}
\end{split}
\end{equation}
\begin{equation*}
\begin{split}
&\leq\sum_{l=\max\{1,j\}}^{\infty}||S_{l-1}\partial_m{\dot{\Delta}}_j{\bar{\omega}}_n\Delta_l v_n||_{L^{\infty}}
\leq C ||{\dot{\Delta}}_j{\bar{\omega}}_n||_{L^{\infty}}||\nabla v_n||_{L^{\infty}},
\end{split}
\end{equation*}
where we applied Bernstein's Lemma and took the sum to get the second inequality.  We bound $||\nabla v_n||_{L^{\infty}}$ as in ($\ref{gradbound}$), we multiply ($\ref{comm2}$) by $2^{-j}$, and we take the supremum over $j\in\Z$.  This yields 
\begin{equation*}
\sup_{j\in\Z}2^{-j}||T_{\partial_m{\dot{\Delta}}_j {\bar{\omega}}_n}v_n(s)||_{L^{\infty}} \leq Cne^{C_1s}||{\bar{\omega}}_n(s)||_{{\dot{B}}^{-1}_{\infty,\infty}}.
\end{equation*}
Moreover, since the Fourier transform of $S_{l-1}\partial_m{\bar{\omega}}_n\Delta_l v_n$ has support in an annulus with inner and outer radius of order $2^l$,  we have for $j\geq 0$
\begin{equation}\label{comm3}
\begin{split}
&||{\dot{\Delta}}_j(T_{\partial_m {\bar{\omega}}_n}v_n)||_{L^{\infty}} = ||{\dot{\Delta}}_j (\sum_{l\geq 1}S_{l-1}\partial_m{\bar{\omega}}_n\Delta_l v_n)||_{L^{\infty}}\\
& \leq \sum_{l=\max\{1,j-4\}}^{j+4}\sum_{k\leq l} 2^{2k}2^{-l} ||\Delta_k {\bar{v}}_n||_{L^{\infty}}||\Delta_l \nabla {v}_n||_{L^{\infty}}\leq C2^j ||{\bar{v}}_n||_{L^{\infty}}||\omega^0||_{L^{\infty}},
\end{split}
\end{equation}
where we used Bernstein's Lemma to get the first inequality, and we used Lemma \ref{CZhighfreq} and ($\ref{NSvortbound}$) to get the second inequality.  For the case $j<0$, $||{\dot{\Delta}}_j(T_{\partial_m {\bar{\omega}}_n}v_n)||_{L^{\infty}}$ is identically $0$.  Therefore ($\ref{comm3}$) still holds.  We bound $||{\bar{v}}_n||_{L^{\infty}}$ as in ($\ref{log}$), we multiply ($\ref{comm3}$) by $2^{-j}$, and we take the supremum over $j\in\Z$, which yields
\begin{equation*}
\sup_{j\in\Z}2^{-j}||{\dot{\Delta}}_j(T_{\partial_m {\bar{\omega}}_n}v_n)(s)||_{L^{\infty}} \leq Ce^{C_1s}2^{-n\alpha}+Cn||{\bar{\omega}}_n(s)||_{{\dot B}^{-1}_{\infty,\infty}}
\end{equation*} 
for any fixed $\alpha\in(0,1)$.  

To estimate the remainder, we again expand the commutator and consider each piece separately.  We break $v_n$ into a low-frequency term and high-frequency term, and we consider $||{\dot{\Delta}}_j(\partial_mR((Id-S_0)v_n,{\bar{\omega}}_n))||_{L^{\infty}}$.  We have 
\begin{equation*}
\begin{split}
&\sup_{j\in\Z}2^{-j}||{\dot{\Delta}}_j(\partial_mR((Id-S_0)v_n,{\bar{\omega}}_n))||_{L^{\infty}}\\
&\qquad\leq C\sum_{l}\sum_{i=-1}^1 ||\Delta_{l-i} (Id-S_0)v_n||_{L^{\infty}}||\Delta_l {\bar{\omega}}_n||_{L^{\infty}}\\ 
&\qquad \leq C||{\bar{\omega}}_n||_{{\dot{B}}^{-1}_{\infty,\infty}}\sum_{l}\sum_{i=-1}^1||\Delta_{l-i} \nabla v_n||_{L^{\infty}},\\
\end{split}
\end{equation*}
where we used Bernstein's Lemma to get the first inequality and the second inequality.  We now apply the arguments in ($\ref{gradbound1}$) and ($\ref{gradbound}$) to conclude that 
\begin{equation*}
\sup_{j\in\Z}2^{-j}||{\dot{\Delta}}_j(\partial_mR((Id-S_0)v_n,{\bar{\omega}}_n))(s)||_{L^{\infty}} \leq Ce^{C_1s}n||{\bar{\omega}}_n(s)||_{{\dot{B}}^{-1}_{\infty,\infty}}.
\end{equation*}
To bound the low frequencies, we again apply Bernstein's Lemma and the definition of the remainder term.  We write
\begin{equation*}
\begin{split}
&\sup_{j\in\Z}2^{-j}||{\dot{\Delta}}_j(\partial_mR(S_0v_n,{\bar{\omega}}_n))||_{L^{\infty}} \leq C\sum_{l\leq 1}\sum_{i=-1}^1 ||\Delta_{l-i} S_0 v_n||_{L^{\infty}}||\Delta_l {\bar{\omega}}_n||_{L^{\infty}}\\
&\qquad\qquad\qquad \leq C||v||_{L^{\infty}}||{\bar{v}}_n||_{L^{\infty}}.\\
\end{split}
\end{equation*}  
The last inequality follows from the bound $||\Delta_l{\bar{\omega}}_n||_{L^{\infty}}\leq||\Delta_l\nabla\bar{ v}||_{L^{\infty}}$, Bernstein's Lemma, and the observation that $l\leq 1$.  We now bound $||v||_{L^{\infty}}$ using ($\ref{Eunifbd}$) and we bound $||{\bar{v}}_n||_{L^{\infty}}$ as in ($\ref{log}$).  This yields
\begin{equation*}
\begin{split}
&\sup_{j\in\Z}2^{-j}||{\dot{\Delta}}_j(\partial_mR(S_0v_n,{\bar{\omega}}_n))(s)||_{L^{\infty}}\\
&\qquad\qquad\leq Ce^{C_1s}2^{-n\alpha}+Ce^{C_1s}n||{\bar{\omega}}_n(s)||_{{\dot B}^{-1}_{\infty,\infty}}.
\end{split}
\end{equation*}
It remains to bound $\sup_{j\in\Z}2^{-j}||\partial_mR(v_n,{\dot{\Delta}}_j{\bar{\omega}}_n))||_{L^{\infty}}$.  Again we break $v_n$ into a low-frequency and high-frequency term.  We first estimate the high-frequency term.  We reintroduce the sum over $m$ and utilize the divergence-free property of $v$ to put the partial derivative $\partial_m$ on ${\bar{\omega}}_n$.  We then apply Bernstein's Lemma to conclude that for any fixed $j\in\Z$
\begin{equation}\label{comm7}
\begin{split}
&\sum_m||R((Id-S_0)v^m_n,{\dot{\Delta}}_j\partial_m{\bar{\omega}}_n))||_{L^{\infty}}\\ 
&\leq C\sum_{|k-l|\leq 1} 2^{l-k}||\Delta_k\nabla v_n||_{L^{\infty}} ||{\dot{\Delta}}_j\Delta_l{\bar{\omega}}_n||_{L^{\infty}}\leq C||\nabla v_n||_{L^{\infty}}||{\dot{\Delta}}_j{\bar{\omega}}_n||_{L^{\infty}}.
\end{split}
\end{equation}
The second inequality above follows because for fixed $j\geq 0$, we are summing only over $l$ satisfying $|l-j|\leq 1$, while for fixed $j<0$, we are only considering $l$ satisfying $-1\leq l \leq 1$.  We now bound $||\nabla v_n||_{L^{\infty}}$ as in ($\ref{gradbound}$), we multiply ($\ref{comm7}$) by $2^{-j}$, and we take the supremum over $j\in\Z$, which yields
\begin{equation*}
\sup_{j\in\Z}2^{-j}\sum_m||\partial_mR(v_n,{\dot{\Delta}}_j{\bar{\omega}}_n))(s)||_{L^{\infty}} \leq Ce^{C_1s}n||{\bar{\omega}}_n(s)||_{{\dot B}^{-1}_{\infty,\infty}}.
\end{equation*}
For the low-frequency term, we again use the divergence-free condition on $v_n$ to write
\begin{equation*}
\begin{split}
&\sup_{j\in\Z}2^{-j}\sum_m||\partial_mR(S_0v^m_n,{\dot{\Delta}}_j{\bar{\omega}}_n)||_{L^{\infty}}\\ &\qquad\leq \sup_{j\in\Z}2^{-j}\sum_{|k-l|\leq 1}||\Delta_kS_0 v_n||_{L^{\infty}} 2^{2j}||{\dot{\Delta}}_j\Delta_l{\bar{v}}_n||_{L^{\infty}}\leq C||v||_{L^{\infty}}||{\bar{v}}_n||_{L^{\infty}}.
\end{split}
\end{equation*}
To get the first inequality, we bounded $||{\dot{\Delta}}_j\partial_m{\bar{\omega}}_n||_{L^{\infty}}$ with $||{\dot{\Delta}}_j\partial_m\nabla{\bar{v}}_n||_{L^{\infty}}$ and applied Bernstein's Lemma.  The second inequality follows from the observation that we are only considering $k\leq 1$, and therefore, by properties of our partition of unity, we are only considering $j\leq 3$.   As with previous terms, we use ($\ref{Eunifbd}$) to bound $||v||_{L^{\infty}}$ and we use ($\ref{log}$) to bound $||{\bar{v}}_n||_{L^{\infty}}$.  We conclude that 
\begin{equation*}
\begin{split}
&\sup_{j\in\Z}2^{-j}\sum_m||\partial_mR(S_0v^m_n,{\dot{\Delta}}_j{\bar{\omega}}_n)(s)||_{L^{\infty}}\\
&\leq Ce^{C_1s}2^{-n\alpha}+Ce^{C_1s}n||{\bar{\omega}}_n(s)||_{{\dot B}^{-1}_{\infty,\infty}}. 
\end{split}
\end{equation*}
This completes the proof of Lemma \ref{comm}.
\Obsolete{
We now show that a Serfati solution to ($E$) satisfies the integral equation given in ($\ref{EINT}$).  We prove the following lemma.
\begin{lemma}\label{Easint}
Let $v$ be the unique solution to ($E$) which satisfies the conditions of Theorem \ref{Serfati}.  The the pressure $p$ satisfies the equality 
\begin{equation*}
p(t) = \sum_{i,j=1}^2\Delta^{-1}\partial_i\partial_j (v^iv^j)(t) +C
\end{equation*}
for every $t\geq 0$, where $C$ is an absolute constant.
\end{lemma}
\begin{proof}
Let $p'= \sum_{i,j=1}^2\Delta^{-1}\partial_i\partial_j (v^iv^j)$.  Taking the divergence of ($E$), we see that 
\begin{equation*}
\Delta p = \sum_{i,j=1}^2\partial_i\partial_j (v^iv^j).
\end{equation*}
Therefore, $\Delta p=\Delta p'$, which implies that $p$ and $p'$ differ by a harmonic polynomial $f$.  To complete the proof, it suffices to show that $f$ is a constant.  

Since $v$ is bounded, it follows by properties of Calderon-Zygmund operators that $p'$ belongs to the space $BMO$ (see, for example, \cite{Stein}).  Moreover, by the properties of the pressure $p$ corresponding to the Serfati solution $v$, we have that $\frac{p(x)}{|x|}$ approaches $0$ as $|x|$ approaches infinity.  Assume, for contradiction, that $f$ is a polynomial of degree greater than or equal to one.  
Given $\epsilon>0$, for $|x|$ sufficiently large, $p'(x)$ satisfies
\begin{equation*}
\begin{split}
|p'(x)| &\geq |f(x)|-|p(x)|\\
&\geq |f(x)|-\epsilon|x|.
\end{split} 
\end{equation*}
It follows that for large $x$, $p'$ behaves like a polynomial of degree greater than or equal to one, which contradicts the membership of $p'$ to $BMO$.  Therefore $f$ must be a constant.  This completes the proof.           
\end{proof}}

\end{document}